\documentclass[12pt, reqno]{amsart}

\usepackage{amsmath, amsthm, amscd, amsfonts, amssymb, graphicx, color}
\usepackage[bookmarksnumbered, colorlinks, plainpages]{hyperref}
\input{mathrsfs.sty}
\hypersetup{colorlinks=true,linkcolor=red, anchorcolor=green, citecolor=cyan, urlcolor=red, filecolor=magenta, pdftoolbar=true}

\newtheorem{theorem}{Theorem}[section]
\newtheorem{lemma}[theorem]{Lemma}

\newtheorem{corollary}[theorem]{Corollary}
\theoremstyle{definition}

\theoremstyle{remark}

\numberwithin{equation}{section}

{\normalsize }
\def\u|{|\kern-0.1em|\kern-0.1em|}
\def\U|{\Big|\kern-0.1em\Big|\kern-0.1em\Big|}

\def\~{\hskip-2pt}

\def\<{\langle}
\def\>{\rangle}

\def\VV{\lower-0.1ex\hbox{$\ \begin{matrix}\vee\\[-2ex]\vee\end{matrix}\ $}}
\def\vv{\lower-0.2ex\hbox{$\ \begin{matrix}\wedge\\[-2ex]\wedge\end{matrix}\ $}}

\def\({\left(}
\def\){\right)}

\def\a{\alpha}
\def\b{\beta}

\def\u|{|\kern-0.1em|\kern-0.1em|}
\def\U|{\Big|\kern-0.1em\Big|\kern-0.1em\Big|}

\numberwithin{equation}{section}
\def\s{\sharp}

\def\phi{\varphi}

\newcommand{\NORM}[1]{\left|\!\left|{#1}\right|\!\right|}

\begin{document}


\title[Variants of Ando--Hiai  inequality for deformed means]{Variants of Ando--Hiai type inequalities for deformed means and applications}

\author[M. Kian, M. S. Moslehian, Y. Seo]{Mohsen Kian$^1$, M. S. Moslehian$^2$ \MakeLowercase{and} Yuki Seo$^3$}

\address{$^1$ Department of Mathematics, University of Bojnord, P.O. Box
1339, Bojnord 94531, Iran}
\email{kian@ub.ac.ir }

\address{$^2$ Department of Pure Mathematics, Center of Excellence in Analysis on Algebraic Structures (CEAAS), Ferdowsi University of Mashhad, P. O. Box 1159, Mashhad 91775, Iran.}
\email{moslehian@um.ac.ir; moslehian@yahoo.com}

\address{$^3$ Department of Mathematics Education, Osaka Kyoiku University, Asahigaoka, Kashiwara, Osaka 582-8582, Japan.}
 \email{yukis@cc.osaka-kyoiku.ac.jp}



\begin{abstract}
For an $n$-tuple of positive invertible operators on a Hilbert space, we present some variants of Ando--Hiai type inequalities for deformed means from an $n$-variable operator mean by an operator mean, which is related to the information monotonicity of a certain unital positive linear map. As an application, we investigate the monotonicity of the power mean from the deformed mean in terms of the generalized Kantorovich constants under the operator order. Moreover, we improve the norm inequality for the operator power means related to the Log-Euclidean mean in terms of the Specht ratio.
\end{abstract}

\maketitle

\section{Introduction}

In 2004, Ando et al. \cite{ALM} succeeded in the formulation of the geometric mean for $n\ (\geq 3)$ positive definite matrices, and they showed that it has many required properties as the geometric mean. Yamazaki \cite{Ya0} pointed out that it can be extended to the positive invertible operators on a Hilbert space. Since then, many researchers have studied operator geometric means of $n$ positive invertible operators on a Hilbert space. On the other hand, Moakher \cite{Mo} and then Bhatia and Holbrook \cite{BH} suggested a new definition of the geometric mean for $n$ positive definite matrices by taking the mean to be the unique minimizer of the sum of squares of distances. Computing appropriate derivatives as in \cite{Mo,B} yields that it coincides with the unique positive definite solution of the Karcher equation. The unique solution of the Karcher equation is called the Karcher mean of $n$ positive definite matrices. In 2012, Lim and P\'{a}lfia \cite{LP} constructed a family of matrix power means, each with numerous desirable properties such as monotonicity, that converges to the Karcher mean and showed that these properties are preserved in the limit. Moreover, in 2014, Lawson and Lim \cite{LL2} showed that the Karcher equation has a unique positive invertible solution in the infinite-dimensional setting.

The so-called Golden--Thompson inequality provides a relation between matrix exponential functions. In separate works, Golden \cite{Gol} and Thompson \cite{Th} proved that
 $\mathrm{Tr} e^{H+K}\leq  \mathrm{Tr} e^{H} e^{K}$ for all Hermitian matrices $H,K$.  An extension of this inequality asserts that  $\left\|e^{H+K}\right\|\leq\left\|e^{H/2}e^K e^{H/2}\right\|$ holds   for  every unitarily invariant norm.  As a complementary to the Golden--Thompson inequality,  Ando and Hiai \cite{AH} presented an inequality  for operator means of positive definite matrices.  It is called the Ando--Hiai inequality (see also \cite{FUJ, FK,WAD}): For each $\a \in (0,1]$
\[
A\ \s_{\a}\ B \leq I \qquad \Longrightarrow \qquad A^r\ \s_{\a}\ B^r \leq I \qquad \mbox{for all $r\geq 1$,}
\]
where the weighted geometric mean is defined by
\[
A\ \s_{\a}\ B = A^{1/2}(A^{-1/2}BA^{-1/2})^{\a}A^{1/2}
\]
for positive invertible operators $A$ and $B$.  The Ando--Hiai inequality is not only a significant inequality in the operator theory \cite{FMPS2}, but also plays an essential role in the quantum information theory, quantum statistics, and so on; see \cite{Furui2008, Se}.  The extension of this inequality to the Karcher mean was established by Yamazaki \cite{Ya1}: For each probability vector $\omega=(\omega_1,\ldots,\omega_n)$
\[
G_{\omega}(A_1,\ldots,A_n) \leq I \qquad \Longrightarrow \qquad G_{\omega}(A_1^r,\ldots,A_n^r)\leq I \qquad \mbox{for all $r\geq 1$,}
\]
where the Karcher mean $G_{\omega}(A_1,\ldots,A_n)$ for positive invertible operators $A_1,\ldots,A_n$ is defined to be the unique solution of the Karcher equation
\[
\sum_{j=1}^n \omega_j \log(X^{-1/2}A_jX^{-1/2})=0,
\]
and its modification for the operator power means was also shown by Wada \cite{Wa}.\par
 Recently, Hiai et al. \cite{HSW} utilized a fixed point method to derive      deformed mean of an $n$-variable mean by a $2$-variable operator mean.  They discussed various  Ando--Hiai type inequalities  for deformed means. For more information on the deformed means, the readers are referred to \cite{H2,HL,Ya2}. \par
In this paper, as a continuation of \cite{KMS}, for an $n$-tuple of positive invertible operators on a Hilbert space, we prove some variants of Ando--Hiai type inequalities for the deformed mean from an $n$-variable operator mean by an operator mean, which is related to the information monotonicity of a certain unital positive linear map.
\par As an application, we discuss the monotonicity of the power mean from the deformed mean in terms of the generalized Kantorovich constants under the operator order. Moreover, we improve the norm inequality for the operator power means related to the Log-Euclidean mean in terms of the Specht ratio.

\par
\medskip
\section{Deformed mean}

 Throughout the paper, $\mathbb{B}(\mathcal{H})$ is the $C^*$-algebra of all bounded linear operators on a Hilbert space $\mathcal{H}$; $\mathbb{B}(\mathcal{H})^+$ is the cone of positive operators in $\mathbb{B}(\mathcal{H})$; and ${\Bbb P}={\Bbb P}(\mathcal{H})$ is the set of positive invertible operators in $\mathbb{B}(\mathcal{H})$. For self-adjoint operators $X$ and $Y$, we write $Y\geq X$ (the operator order) if $Y-X$ is positive, and $Y>X$ if $Y-X$ is positive invertible. We denote by $\NORM{X}_{\infty}$ the operator norm of $X\in \mathbb{B}(\mathcal{H})$. Moreover, we denote by $I$   the identity operator on $\mathcal{H}$  and  by {\rm SOT}   the strong operator topology on $\mathbb{B}(\mathcal{H})$.\par
 The notion of $2$-variable operator means was introduced by Kubo and Ando \cite{KA} in an axiomatic way as follows: A map $\sigma:\mathbb{B}(\mathcal{H})^+ \times \mathbb{B}(\mathcal{H})^+ \mapsto \mathbb{B}(\mathcal{H})^+$ is called an operator mean if it satisfies the following properties:
\begin{enumerate}
\item[(i)] Monotonicity: $A\leq C, B\leq D \Longrightarrow A\ \sigma\ B \leq C\ \sigma\ D$.
\item[(ii)] Transformer inequality: $C(A\ \sigma\ B)C \leq (CAC)\ \sigma\ (CBC)$ for every $C\in \mathbb{B}(\mathcal{H})^+$.
\item[(iii)] Downward continuity: $A_k\searrow A, B_k\searrow B \Longrightarrow A_k\ \sigma\ B_k \searrow A\ \sigma\ B$, where $A_k \searrow A$ means that $A_1\geq A_2\geq \cdots$ and $A_k\to A$ in SOT.
\item[(iv)] Normalization: $I\ \sigma\ I=I$.
\end{enumerate}
\medskip
 The  most important result of \cite{KA}  gives a one-to-one   correspondence $\sigma \leftrightarrow f$ between   operator means $\sigma$ and   nonnegative operator monotone functions $f$ on $(0,\infty)$ with $f(1)=1$ via
\begin{align*}
& f(x)I=I\ \sigma\ (xI) \qquad \mbox{for $x>0$},\\
& A\ \sigma\ B = A^{1/2}f(A^{-1/2}BA^{-1/2})A^{1/2} \qquad \mbox{for $A,B\in {\Bbb P}.$}
\end{align*}
The corresponding operator monotone function $f$  to $\sigma$ is said to be the representing function of $\sigma$ and is denoted by $f_{\sigma}$.

To extend operator means to   several variables, Hiai et al.\cite{HSW} consider a map  $\mathfrak{M}:{\Bbb P}^n \mapsto {\Bbb P}$  as an $n$-variable operator mean  if it satisfies:
\begin{enumerate}
\item[(I)]\ {\it Monotonicity}: If $A_j,B_j \in {\Bbb P}$ and $A_j \leq B_j$ for $1\leq j\leq n$, then
\[
\mathfrak{M}(A_1,\ldots,A_n) \leq \mathfrak{M}(B_1,\ldots,B_n).
\]
\item[(II)]\ {\it Congruence invariance}: For every $A_1,\ldots,A_n \in {\Bbb P}$ and any invertible $S\in \mathbb{B}(\mathcal{H})$,
\[
S^*\mathfrak{M}(A_1,\ldots,A_n)S=\mathfrak{M}(S^*A_1S,\ldots,S^*A_nS).
\]
In particular, the homogeneity holds: $\mathfrak{M}(tA_1,\ldots,tA_n)=t\mathfrak{M}(A_1,\ldots,A_n)$ for $t>0$.
\item[(III)]\ {\it Monotone continuity}: Let $A_j, A_{jk} \in {\Bbb P}$ for $1\leq j\leq n$ and $k\in {\Bbb N}$. If either $A_{jk}\nearrow A_j$ or $A_{jk}\searrow A_j$ as $k\to \infty$ for each $j$, then
\[
\mathfrak{M}(A_{1k},\ldots,A_{nk}) \to \mathfrak{M}(A_1,\ldots,A_n)\quad \mbox{in {\rm SOT}.}
\]
\item[(IV)]\ {\it Normalized condition}: $\mathfrak{M}(I,\ldots,I)=I$.\\
\end{enumerate}

 Assume that $\sigma$ is a $2$-variable operator mean such that $\sigma$ is not the left trivial mean $\ell$ defined by   $A\ell B=A$. If  $\mathfrak{M}$ is an $n$-variable operator mean, then   the deformed mean $\mathfrak{M}_{\sigma}$ from $\mathfrak{M}$ by $\sigma$ is defined in \cite[Theorem 2.1]{HSW} to be the unique positive solution of the operator equation
\[
X=\mathfrak{M}(X\sigma A_1,\ldots,X\sigma A_n) \qquad \mbox{for $X\in {\Bbb P}$},
\]
 for all positive invertible operators $A_1,\ldots,A_n \in {\Bbb P}$,  or equivalently
\[
I=\mathfrak{M}(f_{\sigma}(X^{-1/2}A_1X^{-1/2}),\ldots,f_{\sigma}(X^{-1/2}A_nX^{-1/2})),
\]
where $f_{\sigma}$ is the representing function of $\sigma$. Then $\mathfrak{M}_{\sigma}:{\Bbb P}^n \mapsto {\Bbb P}$ is an $n$-variable operator mean satisfying (I)--(IV) again. Moreover, if $Y\in {\Bbb P}$ and $Y\leq \mathfrak{M}(Y\sigma A_1,\ldots,Y\sigma A_n)$, then $Y\leq \mathfrak{M}_{\sigma}(A_1,\ldots,A_n)$. If $Y\in {\Bbb P}$ and $Y\geq \mathfrak{M}(Y\sigma A_1,\ldots,Y\sigma A_n)$, then $Y\geq \mathfrak{M}_{\sigma}(A_1,\ldots,A_n)$.

For every $2$-variable operator mean $\sigma$, it associates another operator mean $\sigma^*$ defined by $A\sigma^* B=(A^{-1}\sigma B^{-1})^{-1}$. $\sigma^*$ is called the adjoint mean of $\sigma$. Similarly, for an $n$-variable mean $\mathfrak{M}$, the adjoint mean $\mathfrak{M}^*$   is defined by
\[
\mathfrak{M}^*(A_1,\ldots,A_n)=\mathfrak{M}(A_1^{-1},\ldots,A_n^{-1})^{-1}, \quad A_j\in {\Bbb P}.
\]
Clearly $\mathfrak{M}^*$ is itself a mean satisfying (I)--(IV) and $(\mathfrak{M}_{\sigma})^*=(\mathfrak{M}^*)_{\sigma^*}$   for any operator mean $\sigma \not= \ell$.\par

\medskip

 The following result shows that the deformed mean satisfies the information monotonicity, which is already shown under a more general setting by Hiai and Lim \cite{HL}, and under a unital case by P\'{a}lfia \cite{Pa}.
\begin{theorem} \label{thm-IM}
Let $\Phi:\mathbb{B}(\mathcal{H}) \mapsto \mathbb{B}(\mathcal{K})$ be a normal positive linear map such that $\Phi(I)$ is invertible and let $\sigma$ be a $2$-variable operator mean with $\sigma \not= \ell$, which is used in the related constructions to be sure that the corresponding monotone function $f_\sigma$ is strictly monotone on $(0,\infty)$. If $\mathfrak{M}:{\Bbb P}^n \mapsto {\Bbb P}$ is an $n$-variable operator mean with the information monotonicity
\begin{equation*} 
\Phi(\mathfrak{M}(A_1,\ldots,A_n)) \leq \mathfrak{M}(\Phi(A_1),\ldots,\Phi(A_n)),
\end{equation*}
then the deformed mean $\mathfrak{M}_{\sigma}$ satisfies the information monotonicity:
\begin{equation} \label{eq:im2}
\Phi(\mathfrak{M}_{\sigma}(A_1,\ldots,A_n)) \leq \mathfrak{M}_{\sigma}(\Phi(A_1),\ldots,\Phi(A_n)).
\end{equation}
\end{theorem}


 Let $\omega=(\omega_1,\ldots,\omega_n)$ be a probability vector, that is, $\omega_j\geq 0$ and $\sum_{j=1}^n \omega_j =1$. The weighted arithmetic mean $\mathcal{A}_{\omega}$ and the weighted harmonic mean $\mathcal{H}_{\omega}=(\mathcal{A}_{\omega})^*$ are defined by
\[
\mathcal{A}_{\omega}=\sum_{j=1}^n \omega_j A_j \quad \mbox{and}\quad \mathcal{H}_{\omega}=\left(\sum_{j=1}^n \omega_j A_j^{-1}\right)^{-1}
\]
for $A_1,\ldots,A_n \in {\Bbb P}$. Clearly $\mathcal{A}_{\omega}$ and $\mathcal{H}_{\omega}$ satisfy (I)--(IV).\par
Let  $\mathfrak{M}:{\Bbb P}^n \mapsto {\Bbb P}$ be an $n$-variable operator mean with
\begin{equation} \label{eq:HMA}
\mathcal{H}_{\omega} \leq \mathfrak{M} \leq \mathcal{A}_{\omega}
\end{equation}
for some probability vector $\omega=(\omega_1,\ldots,\omega_n)$. Then for any operator mean $\sigma$, it follows that $\mathfrak{M}_{\sigma}$ satisfies \eqref{eq:HMA} again: In fact, put $X=\mathfrak{M}_{\sigma}(A_1,\ldots,A_n)$. It follows from \cite[(3.3.2)]{H1} that $\a_0:=f_{\sigma}'(1) \in (0,1]$ and so $!_{\a_0}\leq \sigma \leq \nabla_{\a_0}$. Hence, the monotonicity of $\mathfrak{M}$ gives
\[
\mathfrak{M}(X\ !_{\a_0}\ A_1,\ldots,X\ !_{\a_0}\ A_n) \leq X \leq \mathfrak{M}(X\ \nabla_{\a_0}\ A_1,\ldots,X\ \nabla_{\a_0}\ A_n)
\]
and the assumption $\mathcal{H}_{\omega} \leq \mathfrak{M} \leq \mathcal{A}_{\omega}$ implies
\begin{equation} \label{eq:hma}
\mathcal{H}_{\omega} \leq \mathfrak{M}_{\sigma} \leq \mathcal{A}_{\omega}.
\end{equation}

If $\mathfrak{M}$ satisfies \eqref{eq:HMA}, then so does $\mathfrak{M}^*$, and hence $(\mathfrak{M}^*)_{\sigma}$ satisfies \eqref{eq:HMA}, too.\par

Replacing $A_j$ by $A_j^p$ for $p>0$ in \eqref{eq:hma} and taking the logarithm of both sides, we get
\[
\log\left( \sum_{j=1}^n \omega_j A_j^{-p}\right)^{-1/p} \leq \log \mathfrak{M}_{\sigma}(A_1^p,\ldots,A_n^p)^{1/p} \leq \log\left( \sum_{j=1}^n \omega_j A_j^p\right)^{1/p}\qquad \mbox{for all $p>0$},
\]
which gives the Lie--Trotter formula for the deformed mean $\mathfrak{M}_\sigma$ as
\begin{equation} \label{eq:LTF}
\lim_{p\to 0} \mathfrak{M}_{\sigma}(A_1^p,\ldots,A_n^p)^{1/p} = \exp \left( \sum_{j=1}^n \omega_j \log A_j \right).
\end{equation}

Let $\Phi:\mathbb{B}(\mathcal{H}) \mapsto \mathbb{B}(\mathcal{K})$ be a unital positive linear map and let $\omega=(\omega_1,\ldots,\omega_n)$ be a probability vector. Then a unital positive linear map $\Psi_{\Phi,\omega}: \mathbb{B}(\mathcal{H})\oplus \cdots \oplus \mathbb{B}(\mathcal{H}) \mapsto \mathbb{B}(\mathcal{K})$ is defined by
\begin{equation} \label{eq:uplm}
\Psi_{\Phi,\omega}(X)=\sum_{j=1}^n \omega_j \Phi(A_j)
\end{equation}
for $X=A_1\oplus\cdots \oplus A_n$ and $A_1,\ldots,A_n \in {\Bbb P}$. If $\Phi$ is the identity map, then we denote it by $\Psi_{\omega}=\Psi_{id,\omega}$.\par
In the next theorem, by the Mond--Pe\v{c}ari\'{c} method \cite[Chapter 1]{FMPS}, we show a reverse to the information monotonicity \eqref{eq:im2} without the information monotonicity condition for $\mathfrak{M}$:
\begin{theorem} \label{thm-RIM}
Let $A_1,\ldots,A_n \in {\Bbb P}$ such that $mI\leq A_j \leq MI$  \ $(j=1,\ldots,n)$ for some scalars $0<m<M$   and let $\Phi:\mathbb{B}(\mathcal{H}) \mapsto \mathbb{B}(\mathcal{K})$ be a unital positive linear map. Assume that $\mathfrak{M}:{\Bbb P}^n \mapsto {\Bbb P}$ is an $n$-variable operator mean satisfying $\mathcal{H}_{\omega} \leq \mathfrak{M} \leq \mathcal{A}_{\omega}$ for some probability vector $\omega=(\omega_1,\ldots,\omega_n)$. If  $\sigma$ is an operator mean, then
\begin{equation} \label{eq:RIM}
\mathfrak{M}_{\sigma}(\Phi(A_1),\ldots,\Phi(A_n))\leq \a \ \Phi(\mathfrak{M}_{\sigma}(A_1,\ldots,A_n))+\b(m,M,\a)I,
\end{equation}
 for each $\a >0$, where
\begin{eqnarray*}
\b( m,M,\a)=\left\{ \begin{array}{ll}
M+m-2\sqrt{\a Mm} & \ \mbox{if \ $m\leq \sqrt{\a Mm}\leq M$}, \\
(1-\a)m & \ \mbox{if \ $M\leq \sqrt{\a Mm}$}, \\
(1-\a)M & \ \mbox{if \ $\sqrt{\a Mm}\leq m$}. \\
\end{array} \right.
\end{eqnarray*}
\end{theorem}

\begin{proof}
By the convexity of $f(t)=t^{-1}$, we have $A^{-1} \leq -\frac{1}{Mm}A+\frac{M+m}{Mm}I$ and hence
\[
\Phi(A^{-1}) \leq -\frac{1}{Mm}\Phi(A)+\frac{M+m}{Mm}I.
\]
Therefore it follows that
\begin{align*}
 \Phi(A)-\a \ \Phi(A^{-1})^{-1} \leq \Phi(A)-\a \left(-\frac{1}{Mm}\Phi(A)+\frac{M+m}{Mm}I\right)^{-1}.
\end{align*}
Assume that the function $F(t)=t-\a\left( \frac{M+m-t}{Mm}\right)^{-1}$ is defined on $[m,M]$. Then it follows that $F'(t)=0$ has exactly one solution $t_0=M+m-\sqrt{\a Mm}$. From $F''(t_0)=2\sqrt{mM\alpha}>0$, we observe that if $m\leq t_0\leq M$, then $\b(m,M,\a)=\max_{m\leq t\leq M} F(t)=F(t_0)=M+m-2\sqrt{\a Mm}$. Moreover, $F''(t)>0$ implies that if $t_0\geq M$, then $\b(m,M,\a)=F(M)=(1-\a)M$ and if $t_0\leq m$, then $\b(m,M,\a)=f(m)=(1-\a)m$. Hence,
\begin{equation} \label{pn2}
 \Phi(A)-\a \ \Phi(A^{-1})^{-1} \leq \b(m,M,\a)I\quad \text{for each } \a >0.
\end{equation}

Since $\mathcal{H}_{\omega} \leq \mathfrak{M} \leq \mathcal{A}_{\omega}$ and $\sigma$ is an operator mean, we have $\mathcal{H}_{\omega} \leq \mathfrak{M}_{\sigma} \leq \mathcal{A}_{\omega}$ and so
\begin{align*}
 \mathfrak{M}_{\sigma}(\Phi(A_1),\ldots,\Phi(A_n))&- \a\ \Phi(\mathfrak{M}_{\sigma}(A_1,\ldots,A_n))\\
& \leq \sum_{j=1}^n \omega_j \Phi(A_j)-\a \ \Phi\left(\left(\sum_{j=1}^n \omega_j A_j^{-1}\right)^{-1}\right) \\
& \leq\sum_{j=1}^n \omega_j \Phi(A_j)-\a\ \left(\sum_{j=1}^n \omega_j \Phi\left(A_j^{-1}\right)\right)^{-1} \\
& = \Psi_{\Phi,\omega}(X)-\a \Psi_{\Phi,\omega}(X^{-1})^{-1} \\
& \leq \b(m,M,\a)I \qquad \mbox{by \eqref{pn2}},
\end{align*}
where $X=A_1\oplus \cdots \oplus A_n$ and $\Psi_{\Phi,\omega}$ is defined by \eqref{eq:uplm}.
\end{proof}

In particular, we have the following ratio type and difference type reverse inequalities of \eqref{eq:im2}.
\begin{corollary} \label{cor-1}
Under the assumptions as in Theorem~\ref{thm-RIM}, it follows that
\begin{equation} \label{eq:R}
\mathfrak{M}_{\sigma}(\Phi(A_1),\ldots,\Phi(A_n))\leq \frac{(M+m)^2}{4Mm} \Phi(\mathfrak{M}_{\sigma}(A_1,\ldots,A_n))
\end{equation}
and
\begin{equation} \label{eq:D}
\mathfrak{M}_{\sigma}(\Phi(A_1),\ldots,\Phi(A_n))\leq \Phi(\mathfrak{M}_{\sigma}(A_1,\ldots,A_n)) + \left( \sqrt{M}-\sqrt{m}\right)^2I.
\end{equation}
\end{corollary}
\begin{proof}
If we choose $\a$ such that $\b(m,M,\a)=0$ in \eqref{eq:RIM}, then $\a$ coincides with the Kantorovich constant $\frac{(M+m)^2}{4Mm}$ and we have \eqref{eq:R}. If we put $\a =1$ in \eqref{eq:RIM}, then $\b(m,M,1)=(\sqrt{M}-\sqrt{m})^2$ and we have \eqref{eq:D}.
\end{proof}
\par
\medskip

\section{Ando--Hiai type inequalities}

 Assume that $\mathfrak{M}:{\Bbb P}^n \mapsto {\Bbb P}$ is an $n$-variable operator mean with conditions (I)--(IV) in previous section.  For $A_1,\ldots,A_n \in {\Bbb P}$ we consider Ando--Hiai type inequalities for $\mathfrak{M}$ as follows:
\begin{equation} \label{eq:AHM}
\mathfrak{M}(A_1^r,\ldots,A_n^r) \leq \NORM{\mathfrak{M}(A_1,\ldots,A_n)}_{\infty}^{r-1}\mathfrak{M}(A_1,\ldots,A_n) \quad \mbox{for all $r\geq 1$},
\end{equation}
and
\begin{equation} \label{eq:AHM2}
\mathfrak{M}(A_1^r,\ldots,A_n^r) \geq \NORM{\mathfrak{M}(A_1,\ldots,A_n)}_{\infty}^{r-1}\mathfrak{M}(A_1,\ldots,A_n) \quad \mbox{for all $0<r\leq 1$}.
\end{equation}
 For example, it follows from \cite[Lemma 3.2 and Corollary 3.4]{HSW} that the weighted harmonic mean $\mathcal{H}_{\omega}$ and the Karcher mean $G_{\omega}$ satisfy \eqref{eq:AHM}.\par
 A two variable operator mean  $\sigma$ is called power monotone increasing (p.m.i)
 if $f_{\sigma}(x^r)\geq f_{\sigma}(x)^r$ for all $x>0$ and $r\geq 1$, see \cite{Wa}. \par
 In \cite[Theorem 3.1]{HSW}, Hiai et al. showed several  Ando--Hiai type inequalities for $n$-variable operator means of   operators: Let $\sigma$ be a p.m.i operator mean with $\sigma \not= \ell$ and let $\mathfrak{M}:{\Bbb P}^n \mapsto {\Bbb P}$ be an $n$-variable operator mean. If $\mathfrak{M}$ satisfies \eqref{eq:AHM} (resp. \eqref{eq:AHM2}) for every $A_1,\ldots,A_n \in {\Bbb P}$, then $\mathfrak{M}_{\sigma}$ satisfies too.\par
 Though we have no relation between $\mathfrak{M}_{\sigma}(A_1^r,\ldots,A_n^r)$ and $\mathfrak{M}_{\sigma}(A_1,\ldots,A_n)^r$ under the operator order for all $r>0$ in general, Ando--Hiai type inequalities induce the following norm inequality: If $\mathfrak{M}$ satisfies \eqref{eq:AHM} (resp. \eqref{eq:AHM2}), then for any p.m.i operator means with $\sigma \not= \ell$, we have
\[
\NORM{\mathfrak{M}_{\sigma}(A_1^r,\ldots,A_n^r)}_{\infty} \leq \NORM{\mathfrak{M}_{\sigma}(A_1,\ldots,A_n)}_{\infty}^r \quad \mbox{for all $r\geq 1$}.
\]
(resp.
\[
\NORM{\mathfrak{M}_{\sigma}(A_1^r,\ldots,A_n^r)}_{\infty} \geq \NORM{\mathfrak{M}_{\sigma}(A_1,\ldots,A_n)}_{\infty}^r \quad \mbox{for all $0<r\leq 1$}.)
\]
 In this section, for every unital positive linear map $\Phi$ and every operator mean $\sigma$, we estimate the operator order relations between two deformed means $\mathfrak{M}_{\sigma}(\Phi(A_1^r),\ldots,\Phi(A_n^r))$ and $\Phi(\mathfrak{M}_{\sigma}(A_1,\ldots,A_n)^r)$ for $r>0$ in terms of the generalized Kantorovich constant.\par

To give our main results, we need some preliminaries. We recall that a unital positive linear mapping $\Phi$ satisfies the Davis--Choi--Jensen type inequality
\begin{align}\label{pnn}
 \Phi(X^r)\leq \Phi(X)^r
\end{align}
for every $r\in(0,1)$ and $X\in\mathbb{P}$. If $r\in(-1,0)\cup(1,2)$, then a reverse inequality holds in \eqref{pnn}. Counterpart to this, we recall the next lemma, which shows some upper and lower bounds for the difference and ratio of Jensen type inequalities for power functions.
\begin{lemma}{\cite[Theorem 3.16]{FMPS}} \label{lem-1}
Let $\Phi:\mathbb{B}(\mathcal{H})\mapsto \mathbb{B}(\mathcal{K})$ be a unital positive linear map and let $X\in {\Bbb P}$ be a positive invertible operator such that $mI \leq X \leq MI$ for some scalars $0<m<M$. If $r\in {\Bbb R}\backslash [0,1]$\ $($resp. $r\in (0,1)$ $)$, then
\[
\Phi(X^r)-\a \Phi(X)^r\leq \gamma(m,M,r,\a)I \quad (resp.\quad \Phi(X^r)-\a \Phi(X)^r \geq \gamma(m,M,r,\a)I)
\]
for each $\a >0$, where
\begin{eqnarray} \label{eq:rmMa}
\gamma(m,M,r,\a)=\left\{ \begin{array}{ll}
\a (r-1)\left( \frac{M^r-m^r}{\a r(M-m)}\right)^{r/(r-1)} +\frac{Mm^r-mM^r}{M-m} & \\
\qquad \qquad \qquad \qquad \quad \mbox{if} \quad m\leq \left( \frac{M^r-m^r}{\a r(M-m)}\right)^{1/(r-1)}\leq M, & \\
 (1-\a)M^r \qquad \qquad \mbox{if}\quad M\leq \left( \frac{M^r-m^r}{\a r(M-m)}\right)^{1/(r-1)}, & \\
(1-\a)m^r \qquad \qquad \mbox{if} \quad \left( \frac{M^r-m^r}{\a r(M-m)}\right)^{1/(r-1)}\leq m. & \\
\end{array} \right.
\end{eqnarray}
\end{lemma}
Two special cases of Lemma \ref{lem-1} read as the next lemma. If $\a =1$, then we write $\gamma(m,M,r)=\gamma(m,M,r,1)$. Moreover if we  choose $\a$ such that $\gamma(m,M,r,\a)=0$ in \eqref{eq:rmMa}, then $\a=K(h,r)$, where
\begin{equation} \label{eq:K}
K(h,p) = \frac{h^p-h}{(p-1)(h-1)} \left( \frac{p-1}{p} \frac{h^p-1}{h^p-h} \right)^p\qquad \mbox{for $p\in {\Bbb R}$ and $h>0$},
\end{equation}
is called the generalized Kantorovich constant and $h=\frac{M}{m}$, see \cite[Definition 2.2]{FMPS}. In particular, $K(h,2)=K(h,-1)=\frac{(h+1)^2}{4h}$ is called the Kantorovich constant.
\begin{lemma}\label{lem-p1}
 With the assumptions as in Lemma \ref{lem-1}, it follows that
\begin{align*}
\Phi(X^r)-\Phi(X)^r \leq \gamma(m,M,r)I \qquad \mbox{and}\qquad \Phi(X^r) \leq K(h,r) \Phi(X)^r
\end{align*}
hold for all $r\in {\Bbb R}\backslash [0,1]$. If $r\in(0,1)$, then reverse inequalities hold.
\end{lemma}


In the next theorem, we give an Ando--Hiai type inequality for deformed means as well as a difference counterpart to the information monotonicity.

\begin{theorem} \label{thm-IMAH}
Assume that $\mathfrak{M}:{\Bbb P}^n \mapsto {\Bbb P}$ is an $n$-variable operator mean satisfies $\mathcal{H}_{\omega} \leq \mathfrak{M}\leq \mathcal{A}_{\omega}$ for some probability vector $\omega=(\omega_1,\ldots,\omega_n)$ and that $\sigma$ is an operator mean. If  $A_1,\ldots,A_n \in {\Bbb P}$ such that $mI\leq A_j \leq MI$  \ $(j=1,\ldots,n)$  for some scalars $0<m<M$, then for each $\a >0$
\begin{equation} \label{eq:imah}
\mathfrak{M}_{\sigma}(\Phi(A_1^r),\ldots,\Phi(A_n^r)) \leq \a \Phi(\mathfrak{M}_{\sigma}(A_1,\ldots,A_n)^r) + \gamma(1/M,1/m,-r,\a)I
\end{equation}
for all $r\in (0,1]$ and all unital positive linear maps   $\Phi$, where $\gamma(m,M,r,\a)$ is defined by \eqref{eq:rmMa}.
\end{theorem}

\begin{proof}
Since $f_{\sigma}'(1)\in (0,1]$ and $\mathcal{H}_{\omega} \leq \mathfrak{M} \leq \mathcal{A}_{\omega}$, we have $\mathcal{H}_{\omega} \leq \mathfrak{M}_{\sigma} \leq \mathcal{A}_{\omega}$. Since $t\mapsto t^r$ is operator monotone for $r\in (0,1]$, we get $\mathcal{H}_{\omega}^r \leq \mathfrak{M}_{\sigma}^r$ and so
\begin{align}\label{pn3}
 \Phi(\mathfrak{M}_{\sigma}^r)\geq \Phi\left(\mathcal{H}_{\omega}^r\right)=\Phi\left(\left(\sum_{j=1}^n \omega_jA_j^{-1}\right)^{-r}\right)\geq \Phi\left(\sum_{j=1}^n \omega_jA_j^{-1}\right)^{-r}=\left(\sum_{j=1}^n \omega_j\Phi(A_j^{-1})\right)^{-r},
\end{align}
where the right inequality follows from the operator convexity of $t\mapsto t^{-r}$ for $-r\in [-1,0)$. Then it follows that
\begin{align*}
 &\mathfrak{M}_{\sigma}(\Phi(A_1^r),\ldots,\Phi(A_n^r)) - \a \Phi(\mathfrak{M}_{\sigma}(A_1,\ldots,A_n)^r) \\
 &\qquad\qquad\qquad \leq \sum_{j=1}^n \omega_j \Phi(A_j^r)-\a \Phi(\mathfrak{M}_{\sigma}(A_1,\ldots,A_n)^r) \qquad (\mbox{by $\mathfrak{M}_{\sigma} \leq \mathcal{A}_{\omega}$})\\
&\qquad\qquad\qquad \leq \sum_{j=1}^n \omega_j \Phi\left((A_j^{-1})^{-r}\right)- \a \left(\sum_{j=1}^n \omega_j\Phi(A_j^{-1})\right)^{-r} \qquad (\mbox{by \eqref{pn3}})\\
&\qquad\qquad\qquad = \Psi_{\Phi,\omega}((X^{-1})^{-r})-\a \Psi_{\Phi,\omega}(X^{-1})^{-r} \\
&\qquad\qquad\qquad \leq \gamma(1/M,1/m,-r,\a)I \qquad\qquad (\mbox{by Lemma~\ref{lem-1} and $-r\in [-1,0)$}),
\end{align*}
where $X=A_1\oplus \cdots \oplus A_n$ and $\frac{1}{M}\leq X^{-1}\leq \frac{1}{m}$, and $\Psi_{\Phi,\omega}$ is defined by \eqref{eq:uplm}.
\end{proof}


 If we put $\a =1$ and choose $\a$ such that $\gamma(m,M,-r,\a)=0$ in Theorem~\ref{thm-IMAH}, then we have the following difference type and ratio type inequalities of \eqref{eq:imah}:

\begin{corollary} \label{cor-dr}
Under the same assumption of Theorem~\ref{thm-IMAH}, if $r\in (0,1)$, then
\begin{equation} \label{eq:d1}
\mathfrak{M}_{\sigma}(\Phi(A_1^r),\ldots,\Phi(A_n^r)) - \Phi(\mathfrak{M}_{\sigma}(A_1,\ldots,A_n)^r) \leq \gamma(1/M,1/m,-r)I
\end{equation}
and
\begin{align} \label{eq:Km}
\frac{4Mm}{(M+m)^2} K(h,-r)^{-1}& \Phi(\mathfrak{M}_{\sigma}(A_1,\ldots,A_n)^r)\leq \mathfrak{M}_{\sigma}(\Phi(A_1^r),\ldots,\Phi(A_n^r)) \\
& \leq K(h,-r) \Phi(\mathfrak{M}_{\sigma}(A_1,\ldots,A_n)^r),\notag
\end{align}
where the generalized Kantorovich constant $K(h,r)$ is defined by \eqref{eq:K} and $h=~M/m$.
\end{corollary}

\begin{proof}
If we put $\a =1$ in \eqref{eq:imah}, then we have \eqref{eq:d1}. If we choose $\a$ such that \\
$\gamma(1/M,1/m,-r,\a)=0$ in \eqref{eq:imah}, then $\a = K(h,-r)$ since $\frac{1/m}{1/M}=M/m=h$ and we obtain the second inequality of \eqref{eq:Km}.

To get the first inequality of \eqref{eq:Km}, note that for every positive invertible operator $A$ with $mI\leq A\leq MI$, we have the reverse Choi's inequality \cite[Theorem 1.32]{FMPS}
\begin{equation} \label{eq:RC}
\Phi(A^{-1})\leq \frac{(M+m)^2}{4Mm}\Phi(A)^{-1}.
\end{equation}
 Furthermore, since $-r\in(-1,0)$, Lemma \ref{lem-p1} for a unital positive linear mapping $\Psi_\omega$ defined by \eqref{eq:uplm} yields that
\begin{align}\label{pn5}
\sum_{j=1}^n \omega_j A_j^{-r} \leq K(h,-r) \left( \sum_{j=1}^n \omega_j A_j\right)^{-r}.
\end{align}
It follows that
{\small\begin{align*}
\mathfrak{M}_{\sigma}(\Phi(A_1^r),\ldots,\Phi(A_n^r)) & \geq \left( \sum_{j=1}^n \omega_j \Phi(A_j^r)^{-1}\right)^{-1}\qquad\qquad\qquad\qquad(\mbox{by $\mathfrak{M}_\sigma\geq \mathcal{H}_\omega$}) \\
& \geq \left( \sum_{j=1}^n \omega_j \Phi(A_j^{-r})\right)^{-1} = \left( \Phi\left(\sum_{j=1}^n \omega_j A_j^{-r}\right)\right)^{-1}\qquad(\mbox{by \eqref{pnn}}) \\
& \geq K(h,-r)^{-1}\Phi\left( \left(\sum_{j=1}^n \omega_j A_j\right)^{-r}\right)^{-1}\qquad\qquad\qquad(\mbox{by \eqref{pn5}}) \\
& \geq K(h,-r)^{-1} \frac{4Mm}{(M+m)^2} \Phi\left( \left(\sum_{j=1}^n \omega_j A_j\right)^{r}\right)\qquad(\mbox{by \eqref{eq:RC}})\\
& \geq \frac{4Mm}{(M+m)^2} K(h,-r)^{-1}\Phi(\mathfrak{M}_{\sigma}(A_1,\ldots,A_n)^r)
\end{align*}}
and the last inequality comes from $\mathcal{A}_\omega\geq \mathfrak{M}_\sigma$ and so $\mathcal{A}_\omega^r \geq \mathfrak{M}_\sigma^r$ for $r\in (0,1)$.
\end{proof}
The L\"{o}wner--Heinz theorem states that if $A\geq B\geq 0$, then $A^p\geq B^p$ for all $p\in [0,1]$. However, $A\geq B$ does not imply    $A^p\geq B^p$ for $p>1$ in general. Related to the Kantorovich inequality, Furuta \cite{F1998} showed the following order preserving operator inequality: Let $A$ and $B$ be positive operators with $A\geq B\geq 0$ and $M_1I\geq A\geq m_1I$ or $M_2I\geq B\geq m_2I$ for some scalars $0<m_1\leq M_1$ and $0<m_2\leq M_2$, and put $h_1=M_1/m_1$ and $h_2=M_2/m_2$. Then
\begin{equation} \label{eq:f1}
B^p \leq K(h_1,p) A^p \qquad \mbox{for all $p\geq 1$}
\end{equation}
and
\begin{equation} \label{eq:f2}
B^p \leq K(h_2,p) A^p \qquad \mbox{for all $p\geq 1$.}
\end{equation}

In the next theorem, we present the Ando--Hiai type inequality for deformed means, when $r\geq 1$.
\begin{theorem}\label{k1}
 Let $A_1,\ldots,A_n \in {\Bbb P}$ such that $mI\leq A_j \leq MI$ \ $(j=1,\ldots,n)$ for some scalars $0<m<M$ and let $\Phi$ be a unital positive linear map. Assume that $\mathfrak{M}:{\Bbb P}^n \mapsto {\Bbb P}$ is an $n$-variable operator mean such that $\mathcal{H}_{\omega} \leq \mathfrak{M} \leq \mathcal{A}_{\omega}$ for some probability vector $\omega=(\omega_1,\ldots,\omega_n)$, and $\sigma$ is an operator mean. Then for each $\a >0$
\begin{equation} \label{eq:abr1}
\mathfrak{M}_\sigma(\Phi(A_1^r),\dots,\Phi(A_k^r))\leq \a \Phi\left(\mathfrak{M}_\sigma(A_1,\ldots,A_k)^r\right) + \gamma\left(\frac{1}{M},\frac{1}{m},-r,\a K(h,r)^{-1} \right)I
\end{equation}
for all $r\geq 1$, where $\gamma(m,M,r,\a)$ is defined by \eqref{eq:rmMa}.
\end{theorem}

\begin{proof}
Since $f_{\sigma}'(1)\in (0,1]$ and $\mathcal{H}_{\omega} \leq \mathfrak{M} \leq \mathcal{A}_{\omega}$, we have $\mathcal{H}_\omega\leq \mathfrak{M}_\sigma \leq\mathcal{A}_\omega$ and $m I\leq \mathfrak{M}_\sigma \leq MI,$ and so it follows from \eqref{eq:f1} that
 \begin{align*} 
\mathcal{H}_\omega^r\leq K(h,r) \mathfrak{M}_\sigma(A_1,\dots,A_k)^r
 \end{align*}
 for all $r\geq 1$. Therefore, for each $\a >0$
\begin{align*}
\sum_{i=1}^k \omega_i A_i^r - \a \mathfrak{M}_\sigma(A_1,\dots,A_k)^r & \leq \sum_{i=1}^k \omega_i A_i^r - \a K(h,r)^{-1} \mathcal{H}_\omega^r.
\end{align*}
Moreover, Lemma~\ref{lem-1} with $X=A_1^{-1}\oplus\cdots\oplus A_k^{-1}$ gives
\begin{align*}
\sum_{i=1}^k \omega_i A_i^r -\a K(h,r)^{-1} \mathcal{H}_\omega^r & = \sum_{i=1}^k \omega_i \left(A_i^{-1}\right)^{-r} - \alpha K(h,r)^{-1} \left( \sum_{i=1}^k \omega_i A_i^{-1} \right)^{-r}\\
 &=\Psi_\omega(X^{-r})-\alpha K(h,r)^{-1} \Psi_\omega(X)^{-r}\\
 &\leq \gamma\left(\frac{1}{M},\frac{1}{m},-r,\alpha K(h,r)^{-1} \right)I,
\end{align*}
where $\Psi_{\omega}$ is defined by \eqref{eq:uplm}. Combining these two inequalities yields
\[
\sum_{i=1}^k \omega_i A_i^r - \alpha \mathfrak{M}_\sigma(A_1,\dots,A_k)^r\leq \gamma\left(\frac{1}{M},\frac{1}{m},-r,\alpha K(h,r)^{-1} \right)I.
\]
Hence, for each $\a >0$
{\small\begin{align*}
\mathfrak{M}_\sigma(\Phi(A_1^r),\dots,\Phi(A_k^r)) - \a \Phi\left(\mathfrak{M}_{\sigma}(A_1,\dots,A_k)^r\right)
& \leq \sum_{j=1}^n \omega_j \Phi(A_j^r) - \a \Phi\left(\mathfrak{M}_{\sigma}(A_1,\dots,A_k)^r\right)\\
& = \Phi\left( \sum_{j=1}^n \omega_j A_j^r - \a \mathfrak{M}_{\sigma}(A_1,\dots,A_k)^r \right)\\
& \leq \gamma\left(\frac{1}{M},\frac{1}{m},-r,\a K(h,r)^{-1} \right)I\,.
\end{align*}}
\end{proof}

\begin{corollary}\label{co:36}
Under the same assumption of Theorem~\ref{k1}, if $r\geq 1$, then
 \begin{equation} \label{eq:k1-1}
 \mathfrak{M}_{\sigma}(\Phi(A_1^r),\dots,\Phi(A_k^r))-\Phi\left(\mathfrak{M}_{\sigma}(A_1,\dots,A_k)^r\right) \leq \gamma\left(\frac{1}{M},\frac{1}{m},-r,K(h,r)^{-1} \right)I
 \end{equation}
 and
\begin{align} \label{eq:k1-2}
& \left(\frac{(M^r+m^r)^2}{4M^rm^r}\right)^{-1} K(h,-r)^{-1}K(h,r)^{-1} \Phi(\mathfrak{M}_{\sigma}(A_1,\dots,A_k)^r) \notag \\
&\qquad \leq \mathfrak{M}_{\sigma}(\Phi(A_1^r),\dots,\Phi(A_k^r))\leq K(h,-r)K(h,r)\ \Phi(\mathfrak{M}_{\sigma}(A_1,\dots,A_k)^r).
 \end{align}
\end{corollary}

\begin{proof}
If we put $\a =1$ in \eqref{eq:abr1} of Theorem~\ref{k1}, then we have \eqref{eq:k1-1}. Moreover, from the discussion before Lemma \ref{lem-p1}, we know that if $\gamma\left(\frac{1}{M},\frac{1}{m},-r,\a K(h,r)^{-1} \right)=0$, then  $\a K(h,r)^{-1}=K(h,-r)$ and so $\a = K(h,r) K(h,-r)$. Hence, we obtain the second inequality of \eqref{eq:k1-2} by putting $\gamma\left(\frac{1}{M},\frac{1}{m},-r,\a K(h,r)^{-1} \right)=0$ in \eqref{eq:abr1} of Theorem~\ref{k1}. \par
Next, we prove the first inequality of \eqref{eq:k1-2}. It follows from Lemma~\ref{lem-p1} with $X=A_1\oplus \cdots \oplus A_n$ and $mI\leq X\leq MI$ that
\[
\sum_{j=1}^n \omega_j A_j^{-r}=\Psi_{\omega}(X^{-r})\leq K(h,-r)\Psi_{\omega}(X)^{-r}=K(h,-r)\left(\sum_{j=1}^n \omega_j A_j\right)^{-r}
\]
for $-r\leq -1$, and so
\begin{equation*} 
\left( \sum_{j=1}^n \omega_j A_j^{-r}\right)^{-1} \geq K(h,-r)^{-1}\left( \sum_{j=1}^n \omega_j A_j\right)^r.
\end{equation*}
Applying the  reverse Choi's inequality \eqref{eq:RC} with $X=A_1^{-r}\oplus \cdots \oplus A_n^{-r}$ and noting that $M^{-r}I\leq X\leq m^{-r}I$, we get
\[
\Psi_{\omega}(X^{-1})\leq \frac{(M^{-r}+m^{-r})^2}{4M^{-r}m^{-r}}\Psi_{\omega}(X)^{-1} = \frac{(M^{r}+m^{r})^2}{4M^{r}m^{r}}\Psi_{\omega}(X)^{-1},
\]
where $\Psi_{\omega}$ is defined by \eqref{eq:uplm} and so
\begin{equation} \label{eq:2}
\left( \Phi\left(\sum_{j=1}^n \omega_j A_j^{-r}\right)\right)^{-1} \geq \left( \frac{(M^{r}+m^{r})^2}{4M^{r}m^{r}}\right)^{-1} \Phi\left(\left(\sum_{j=1}^n \omega_j A_j^{-r}\right)^{-1}\right).
\end{equation}
On the other hand, by Choi's inequality, we can write
\begin{align}\label{eq:2112}
\left(\Phi\left(\sum_{j=1}^n \omega_j A_j^{-r}\right)\right)^{-1}= \left( \sum_{j=1}^n \omega_j \Phi(A_j^{-r})\right)^{-1} \leq \left( \sum_{j=1}^n \omega_j \Phi(A_j^r)^{-1}\right)^{-1}.
\end{align}

Hence, we have
{\small\begin{align*}
\mathfrak{M}_{\sigma}(\Phi(A_1^r),\ldots,\Phi(A_n^r)) & \geq \left( \sum_{j=1}^n \omega_j \Phi(A_j^r)^{-1}\right)^{-1} \\
& = \left(\Phi\left(\sum_{j=1}^n \omega_j A_j^{-r}\right)\right)^{-1}\qquad(\mbox{by \eqref{eq:2112}}) \\
& \geq \left( \frac{(M^{r}+m^{r})^2}{4M^{r}m^{r}}\right)^{-1} \Phi\left(\left(\sum_{j=1}^n \omega_j A_j^{-r}\right)^{-1}\right)\qquad(\mbox{by \eqref{eq:2}})\\
& \geq \left( \frac{(M^{r}+m^{r})^2}{4M^{r}m^{r}}\right)^{-1} K(h,-r)^{-1} \Phi\left(\left(\sum_{j=1}^n \omega_j A_j\right)^r \right)\quad(\mbox{by Lemma \ref{lem-p1}}) \\
& \geq \left( \frac{(M^{r}+m^{r})^2}{4M^{r}m^{r}}\right)^{-1} K(h,-r)^{-1} K(h,r)^{-1}\Phi(\mathfrak{M}_{\sigma}(A_1,\ldots,A_n)^r),
\end{align*}}
whence we derive the first inequality of \eqref{eq:k1-2}. Note that the last inequality comes from $\mathfrak{M}_{\sigma}\leq \mathcal{A}_\omega$ and so $\mathfrak{M}_{\sigma}^r\leq K(h,r) \mathcal{A}_\omega^r$ for all $r\geq1$ by \eqref{eq:f1}.
\end{proof}


It is known that $\mathfrak{M}_{\sigma}(A_1^r,\ldots,A_n^r)$ and $\mathfrak{M}_{\sigma}(A_1,\ldots,A_n)^r$ have no relation under the operator order for $r\in {\Bbb R}$. As an application of our results, we have the following corollary.


\begin{corollary} \label{cor-ahr}
 Let $A_1,\ldots,A_n \in {\Bbb P}$ such that $mI\leq A_j \leq MI$ \ $(j=1,\ldots,n)$ for some scalars $0<m<M$. If
 $\mathfrak{M}:{\Bbb P}^n \mapsto {\Bbb P}$ is an $n$-variable operator mean such that $\mathcal{H}_{\omega} \leq \mathfrak{M}_{\sigma} \leq \mathcal{A}_{\omega}$ for some probability vector $\omega=(\omega_1,\ldots,\omega_n)$, then
{\small \begin{align} \label{eq:khr}
K(h,-r)^{-1}\mathfrak{M}_{\sigma}(A_1,\ldots,A_n)^r & \leq \mathfrak{M}_{\sigma}(A_1^r,\ldots,A_n^r) \\
& \leq K(h,-r) \mathfrak{M}_{\sigma}(A_1,\ldots,A_n)^r,\qquad (r\in (0,1)) \notag
\end{align} }
and
 \begin{align}\label{kkk}
 K(h,-r)^{-1}K(h,r)^{-1}\mathfrak{M}_{\sigma}(A_1,\dots,A_k)^r & \leq \mathfrak{M}_{\sigma}(A_1^r,\dots,A_k^r)\\
& \leq K(h,-r)K(h,r)\mathfrak{M}_{\sigma}(A_1,\dots,A_k)^r\qquad (r\geq 1) \notag
 \end{align}
 for every operator mean $\sigma$. In particular
 \[
\mathfrak{M}_{\sigma}(A_1,\ldots,A_n)\leq I \qquad \mbox{implies} \qquad \mathfrak{M}_{\sigma}(A_1^r,\ldots,A_n^r) \leq K(h,-r)I
 \quad (r\in(0,1))
 \]
 and
 $$\mathfrak{M}_{\sigma}(A_1,\dots,A_k)\leq I\quad\mbox{implies }\qquad \mathfrak{M}_{\sigma}(A_1^r,\dots,A_k^r)\leq K(h,-r)K(h,r)I\quad (r\geq 1).$$
\end{corollary}

\begin{proof}
If we put the identity map $\Phi$ in \eqref{eq:Km}, then we have the right inequality in \eqref{eq:khr}. The right Inequality in \eqref{kkk}  follows similarly from Corollary \ref{co:36}.  For the first inequality in \eqref{eq:khr}, note that $\mathcal{H}_{\omega} \leq \mathfrak{M}_{\sigma}\leq \mathcal{A}_{\omega}$ and $\mathfrak{M}_{\sigma}^r\leq \mathcal{A}_{\omega}^r$ for all $r\in(0,1)$. Then
\begin{align*}
\mathfrak{M}_{\sigma}(A_1^r,\ldots,A_n^r) & \geq \left(\sum_{j=1}^n \omega_j A_j^{-r}\right)^{-1}\\
& \geq K(h,-r)^{-1} \left( \sum_{j=1}^n \omega_j A_j\right)^r \qquad \mbox{(by \eqref{pn5}) }\\
& \geq K(h,-r)^{-1}\mathfrak{M}_{\sigma}(A_1,\ldots,A_n)^r.
\end{align*}

 To prove the first inequality in \eqref{kkk}, we have
\begin{align*}
\mathfrak{M}_{\sigma}(A_1^r,\ldots,A_n^r) & \geq \left(\sum_{j=1}^n \omega_j A_j^{-r}\right)^{-1}\\
& \geq K(h,-r)^{-1} \left( \sum_{j=1}^n \omega_j A_j\right)^r \qquad \mbox{(by \eqref{pn5}) }\\
& \geq K(h,-r)^{-1} K(h,r)^{-1} \mathfrak{M}_{\sigma}(A_1,\ldots,A_n)^r \qquad \mbox{(by \eqref{eq:f1})}
\end{align*}
for $r\geq 1$.
\end{proof}

For each $\a \in [-1,1]\backslash \{ 0\}$ the operator power mean $P_{\omega, \a}(A_1,\ldots,A_n)$ is defined as the unique solution to the equation
\begin{align*}
X & = \mathcal{A}_{\omega}(X\s_{\a} A_1,\ldots,X\s_{\a} A_n)\qquad\ \ \mbox{for $0<\a<1$}\\
X & = \mathcal{H}_{\omega}(X\s_{-\a} A_1,\ldots,X\s_{-\a} A_n)\qquad \mbox{for $-1<\a<0$},
\end{align*}
that is, for $0<\a <1$,
\[
P_{\omega, \a}=(\mathcal{A}_{\omega})_{\s_{\a}} \qquad \mbox{and} \qquad P_{\omega, -\a}=(\mathcal{H}_{\omega})_{\s_{\a}}=(P_{\omega, \a})^*.
\]

We can show a reverse Ando--Hiai inequality for the operator power mean $P_{\omega,\a}$ in the traditional way:
\begin{equation} \label{eq:ahk}
P_{\omega, \a}(A_1^r,\ldots,A_n^r) \leq \left( \frac{(M+m)^2}{4Mm}\right)^r P_{\omega, \a}(A_1,\ldots,A_n)^r \qquad \mbox{for $0<r<1$.}
\end{equation}

In fact, for $0<r\leq 1$,
\begin{align*}
P_{\omega,\a}(A_1^r,\ldots,A_n^r) & \leq \sum_{j=1}^n \omega_j A_j^r \leq \left( \sum_{j=1}^n \omega_j A_j \right)^r \\
& \leq \left(\frac{(M+m)^2}{4Mm} \left( \sum_{j=1}^n \omega_j A_j^{-1}\right)^{-1}\right)^{r} \\
& \leq \left(\frac{(M+m)^2}{4Mm}\right)^r P_{\omega,\a}(A_1,\ldots,A_n)^r
\end{align*}
and we have \eqref{eq:ahk}.

If we put $\mathfrak{M}=\mathcal{A}_{\omega}$ and $\sigma=\s_{\a}$ in Corollary~\ref{cor-ahr}, then $\mathfrak{M}_{\sigma}=(\mathcal{A}_{\omega})_{\s_{\a}}=P_{\omega, \a}$ for $0<\a <1$ and we have
\begin{equation} \label{eq:ahk2}
P_{\omega, \a}(A_1^r,\ldots,A_n^r) \leq K(h,-r) P_{\omega, \a}(A_1,\ldots,A_n)^r.
\end{equation}
Then inequality \eqref{eq:ahk2} is an improvement of inequality \eqref{eq:ahk}. To show it, we need the following lemma.
\begin{lemma} \label{lem-kan}
Let $h\geq 1$. Then the generalized Kantorovich constant has the following property:
\[
K(h,-r)\leq K(h,-1)^r \qquad \mbox{for $r\in (0,1)$}
\]
and
\[
K(h,-r)\geq K(h,-1)^r \qquad \mbox{for $r\not\in (0,1)$}
\]
\end{lemma}

\begin{proof}
Put $F(r)=\log K(h,-r)-r\log K(h,-1)$. Since it follows from \cite[Theorem 2.3]{To} that $\log K(h,-r)$ is convex for $r\in {\Bbb R}$, and $F(0)=F(1)=0$, we have $F(r)<0$ for $r\in (0,1)$ and $F(r)>0$ for $r\not\in (0,1)$.
\end{proof}

By Lemma~\ref{lem-kan}, we have $K(h,-r) < \left( \frac{(M+m)^2}{4Mm}\right)^r$ for $0<r<1$, and so inequality \eqref{eq:ahk2} is an improvement of inequality \eqref{eq:ahk}. \par
\bigskip

\section{Norm inequalities for deformed means}

Let $\mathfrak{M}:{\Bbb P}^n \mapsto {\Bbb P}$ be an $n$-variable operator mean satisfying (I)--(IV) introduced in Section 2. Assume that $\mathcal{H}_{\omega} \leq \mathfrak{M} \leq \mathcal{A}_{\omega}$ for some probability vector $\omega=(\omega_1,\ldots,\omega_n)$ and that $\mathfrak{M}$ satisfies \eqref{eq:AHM}. Then the Lie--Trotter formula \eqref{eq:LTF} for the deformed mean $\mathfrak{M}_{\sigma}$ holds, and $\mathfrak{M}_{\sigma}$ satisfies \eqref{eq:AHM} for a p.m.i operator mean $\sigma \not= \ell$. Hence, for every $A_1,\ldots,A_n \in {\Bbb P}$, $\NORM{\mathfrak{M}_{\sigma}(A_1^p,\ldots,A_n^p)}_{\infty}^{1/p}$ increases to $\NORM{\exp\left( \sum_{j=1}^n \omega_j \log A_j \right)}_{\infty}$ as $p\searrow 0$. On the other hand, $\mathfrak{M}_{\sigma}(A_1^p,\ldots,A_n^p)^{1/p}$ for $p>0$ is not monotone increasing under the operator order. In this section, we consider their operator order relations among $\mathfrak{M}_{\sigma}(A_1^p,\ldots,A_n^p)^{1/p}$ for $p>0$ in terms of the Specht ratio and the generalized Kanotorovich constant.\par
 For this, we recall an important constant due to Specht \cite{Specht}, which estimates the upper bound of the arithmetic mean by the geometric one for positive numbers: For $x_1,x_2,\ldots, x_n \in [m,M]$ with $0<m<M$ and $h=M/m$,
\begin{equation*} 
\frac{x_1+x_2+\cdots +x_n}{n} \leq S(h) \sqrt[n]{x_1x_2\cdots x_n},
\end{equation*}
where the Specht ratio is defined by
\begin{equation} \label{eq:sr}
S(h)=\frac{(h-1)h^{\frac{1}{h-1}}}{e\log h} \quad (h\not= 1)\qquad \mbox{and} \qquad S(1)=1.
\end{equation}
It is known in \cite[Theorem 2.56]{FMPS} that $K(h^r,s/r)\mapsto S(h^s)$ as $r\to 0$.\par

 Firstly, we show the following relation among $\mathfrak{M}_{\sigma}(A_1^p,\ldots,A_n^p)^{1/p}$ for $p>0$ under the operator order without the p.m.i condition of $\sigma$:

\begin{theorem}\label{th4-1}
Let $A_1,\ldots,A_n \in {\Bbb P}$ such that $mI\leq A_j \leq MI$ \ $(j=1,\ldots,n)$ for some scalars $0<m<M$. Assume that $\mathfrak{M}:{\Bbb P}^n \mapsto {\Bbb P}$ is an $n$-variable operator mean and satisfies $\mathcal{H}_{\omega} \leq \mathfrak{M} \leq \mathcal{A}_{\omega}$ for some probability vector $\omega=(\omega_1,\ldots,\omega_n)$ and that $\sigma$ is an operator mean.
If $1\leq q \leq p$, then
\begin{align*}
K(h^p,-\frac{q}{p})^{-1/q}\mathfrak{M}_{\sigma}(A_1^p,\ldots,A_n^p)^{1/p} & \leq \mathfrak{M}_{\sigma}(A_1^q,\ldots,A_n^q)^{1/q} \\
& \leq K(h^p,-\frac{q}{p})^{1/q}\mathfrak{M}_{\sigma}(A_1^p,\ldots,A_n^p)^{1/p}.
\end{align*}
If $0<q<1$ and $q<p$, then
\begin{align*}
K(h^q,1/q)^{-1}K(h^p,-q/p)^{-1/q} & \mathfrak{M}_{\sigma}(A_1^p,\ldots,A_n^p)^{1/p} \leq \mathfrak{M}_{\sigma}(A_1^q,\ldots,A_n^q)^{1/q} \\
& \leq K(h^q,1/q) K(h^p,-q/p)^{1/q}\mathfrak{M}_{\sigma}(A_1^p,\ldots,A_n^p)^{1/p},
\end{align*}
where the generalized Kantorovich constant $K(h,p)$ is defined by \eqref{eq:K} and $h=M/m$.
In particular, if $q\to 0$, then
\begin{align*}
S(h)^{-1} S(h^p)^{-1/p} & \mathfrak{M}_{\sigma}(A_1^p,\ldots,A_n^p)^{1/p} \leq \exp\left( \sum_{j=1}^n \omega_j \log A_j \right) \\
& \leq S(h) S(h^p)^{1/p} \mathfrak{M}_{\sigma}(A_1^p,\ldots,A_n^p)^{1/p}
\end{align*}
for all $p>0$, where the Specht ratio $S(h)$ is defined by \eqref{eq:sr}.
\end{theorem}

\begin{proof}
We prove the right inequalities. The left ones can be derived similarly.
Suppose that $1\leq q\leq p$. Since $0<q/p\leq 1$, it follows from Corollary~\ref{cor-ahr} that
\begin{align*}
\mathfrak{M}_{\sigma}(A_1^{q/p},\ldots,A_n^{q/p})\leq K(h,-q/p)\mathfrak{M}_{\sigma}(A_1,\ldots,A_n)^{q/p}.
\end{align*}
 Replacing $A_j$ by $A_j^p$ for all $j=1,\ldots,n$, we get
\begin{align}\label{I}
\mathfrak{M}_{\sigma}(A_1^{q},\ldots,A_n^{q})\leq K(h^p,-q/p)\mathfrak{M}_{\sigma}(A_1^p,\ldots,A_n^p)^{q/p}.
\end{align}
Since $0<1/q\leq 1$, by the L\"{o}wner--Heinz inequality, we have
\[
\mathfrak{M}_{\sigma}(A_1^{q},\ldots,A_n^{q})^{1/q}\leq K(h^p,-q/p)^{1/q}\mathfrak{M}_{\sigma}(A_1^p,\ldots,A_n^p)^{1/p}
\]
as desired.\par
Suppose that $0<q<1$ and $q<p$. By the discussion above, we have
\[
\mathfrak{M}_{\sigma}(A_1^{q},\ldots,A_n^{q})\leq K(h^p,-q/p)\mathfrak{M}_{\sigma}(A_1^p,\ldots,A_n^p)^{q/p}.
\]
Since $1/q\geq 1$, it follows from \eqref{eq:f2} that
\[
\mathfrak{M}_{\sigma}(A_1^{q},\ldots,A_n^{q})^{1/q}\leq K(h^q,1/q) K(h^p,-q/p)^{1/q}\mathfrak{M}_{\sigma}(A_1^p,\ldots,A_n^p)^{1/p}.
\]
If $q\to 0$, then $K(h^q,1/q)\to S(h)$ and $K(h^p,-q/p)^{1/q}\to S(h^p)^{1/p}$, and thus we have the desired inequality by the Lie--Trotter formula \eqref{eq:LTF}.
\end{proof}


 The next theorem gives a monotonicity property for the norm of the deformed means.

\begin{theorem}\label{thpn1}
Let $A_1,\ldots,A_n \in {\Bbb P}$ such that $mI\leq A_j \leq MI$ \ $(j=1,\ldots,n)$ for some scalars $0<m<M$. Assume that $\mathfrak{M}:{\Bbb P}^n \mapsto {\Bbb P}$ is an $n$-variable operator mean and satisfies $\mathcal{H}_{\omega} \leq \mathfrak{M} \leq \mathcal{A}_{\omega}$ for some probability vector $\omega=(\omega_1,\ldots,\omega_n)$ and that $\sigma$ is an operator mean. Then
\begin{align*}
 K(h^p,-q/p)^{-1/q} \NORM{\mathfrak{M}_{\sigma}(A_1^p,\ldots,A_n^p)}_{\infty}^{1/p} & \leq \NORM{\mathfrak{M}_{\sigma}(A_1^q,\ldots,A_n^q)}_{\infty}^{1/q} \\
& \leq K(h^p,-q/p)^{1/q} \NORM{\mathfrak{M}_{\sigma}(A_1^p,\ldots,A_n^p)}_{\infty}^{1/p}
\end{align*}
for all $0<q<p$, where the generalized Kantorovich constant $K(h,p)$ is defined by \eqref{eq:K} and $h=M/m$. In particular, if $q\to 0$, then
\begin{align*}
 S(h^p)^{-1/p} \NORM{\mathfrak{M}_{\sigma}(A_1^p,\ldots,A_n^p)}_{\infty}^{1/p} & \leq \NORM{\exp\left( \sum_{j=1}^n \omega_j \log A_j \right)}_{\infty} \\
& \leq S(h^p)^{1/p} \NORM{\mathfrak{M}_{\sigma}(A_1^p,\ldots,A_n^p)}_{\infty}^{1/p}
\end{align*}
for all $p>0$, where the Specht ratio $S(h)$ is defined by \eqref{eq:sr}.
\end{theorem}

\begin{proof}
For $0<q<p$, we have $0<q/p<1$ and so it follows from  \eqref{I} and the opposite inequality that
\begin{align*}
K(h^p,-q/p)^{-1} \mathfrak{M}_{\sigma}(A_1^p,\ldots,A_n^p)^{q/p} & \leq
\mathfrak{M}_{\sigma}(A_1^q,\ldots,A_n^q)\\
&  \leq K(h^p,-q/p) \mathfrak{M}_{\sigma}(A_1^p,\ldots,A_n^p)^{q/p}
\end{align*}
and so
\begin{align*}
K(h^p,-q/p)^{-1} \NORM{\mathfrak{M}_{\sigma}(A_1^p,\ldots,A_n^p)^{q/p}}_{\infty}^{q/p} & \leq
\NORM{\mathfrak{M}_{\sigma}(A_1^q,\ldots,A_n^q)}_{\infty} \\
& \leq K(h^p,-q/p) \NORM{\mathfrak{M}_{\sigma}(A_1^p,\ldots,A_n^p)}_{\infty}^{q/p}.
\end{align*}
Hence, we arrive at the desired inequality:
\begin{align*}
K(h^p,-q/p)^{-1/q} \NORM{\mathfrak{M}_{\sigma}(A_1^p,\ldots,A_n^p)^{1/p}}_{\infty}^{1/p} & \leq
\NORM{\mathfrak{M}_{\sigma}(A_1^q,\ldots,A_n^q)}_{\infty}^{1/q} \\
& \leq K(h^p,-q/p)^{1/q} \NORM{\mathfrak{M}_{\sigma}(A_1^p,\ldots,A_n^p)}_{\infty}^{1/p}.
\end{align*}
\end{proof}

Let $\mathfrak{M}:{\Bbb P}^n \mapsto {\Bbb P}$ be an $n$-variable operator mean such that $\mathcal{H}_{\omega} \leq \mathfrak{M} \leq \mathcal{A}_{\omega}$ for some probability vector $\omega=(\omega_1,\ldots,\omega_n)$ and let $\sigma$ an operator mean. Let $A_1,\ldots,A_n \in {\Bbb P}$ such that $mI\leq A_j \leq MI$ \ $(j=1,\ldots,n)$ for some scalars $0<m<M$ and $h=M/m$. We remark that, by Theorem \ref{th4-1}, under the operator order it holds that
\begin{align}\label{rf}
S(h)^{-2}\mathfrak{M}_{\sigma}(A_1,\ldots,A_n)\leq \exp\left( \sum_{j=1}^n \omega_j \log A_j \right)\leq S(h)^{2}\mathfrak{M}_{\sigma}(A_1,\ldots,A_n).
\end{align}

On the other hand, under the norm inequality, we have
\begin{equation} \label{rfn}
S(h)^{-1} \NORM{\mathfrak{M}_{\sigma}(A_1,\ldots,A_n)}_{\infty}\leq \NORM{\exp\left( \sum_{j=1}^n \omega_j \log A_j \right)}_{\infty}\leq S(h)\NORM{\mathfrak{M}_{\sigma}(A_1,\ldots,A_n)}_{\infty}.
\end{equation}
Comparing \eqref{rf} and \eqref{rfn} indicates the difference between the operator order and the operator norm in terms of the Specht ratio.\par
If we put $\mathfrak{M}=\mathcal{H}_{\omega}$ and $\sigma=\s_{\a}$ in Theorem \ref{thpn1}, then $\mathfrak{M}_{\sigma}=(\mathcal{H}_{\omega})_{\s_{\a}}=P_{\omega, -\a}$ for $0<\a <1$ and we have
\[
\NORM{P_{\omega, -\a}(A_1^q,\ldots,A_n^q)}_{\infty}^{1/q} \leq K(h^p,-q/p)^{1/q} \NORM{P_{\omega, -\a}(A_1^p,\ldots,A_n^p)}_{\infty}^{1/p}
\]
for all $0<q<p$. If $\a \to 0$, then
\begin{equation} \label{eq:esp}
\NORM{G_{\omega}(A_1^q,\ldots,A_n^q)}_{\infty}^{1/q} \leq K(h^p,-q/p)^{1/q} \NORM{G_{\omega}(A_1^p,\ldots,A_n^p)}_{\infty}^{1/p}
\end{equation}
for all $0<q<p$. Moreover, if $q\to 0$, then
\begin{equation*}
\NORM{\exp\left( \sum_{j=1}^n \omega_j \log A_j \right)}_{\infty} \leq S(h^p)^{1/p} \NORM{G_{\omega}(A_1^p,\ldots,A_n^p)}_{\infty}^{1/p}
\end{equation*}
for all $p>0$. In \cite[Corollary 4.2]{FSY}, it was shown that
\begin{equation} \label{eq:ekp}
\NORM{G_{\omega}(A_1^q,\ldots,A_n^q)}_{\infty}^{1/q} \leq \left(\frac{(M^p+m^p)^2}{4M^pm^p}\right)^{1/p} \NORM{G_{\omega}(A_1^p,\ldots,A_n^p)}_{\infty}^{1/p}
\end{equation}
for all $p>0$.

If $0<q<p$, then Lemma \ref{lem-kan} implies that $K(h^p,-q/p)<K(h^p,-1)^{q/p} = \left(\frac{(M^p+m^p)^2}{4M^pm^p}\right)^{q/p}$ and so
\begin{align}\label{kk}
 K(h^p,-q/p)^{1/q}<\left(\frac{(M^p+m^p)^2}{4M^pm^p}\right)^{1/p}.
\end{align}
 This ensures that  inequality \eqref{eq:esp} is an improvement of  \eqref{eq:ekp}.  In addition, note that when $q\to 0$, inequality \eqref{kk} implies that
 $S(h^p)^{1/p}< \left(\frac{(M^p+m^p)^2}{4M^pm^p}\right)^{1/p}$ for all $p>0$.\par
\bigskip

\textbf{Acknowledgement.} The third author is partially supported by the Ministry of Education, Science, Sports and Culture, Grant-in-Aid for Scientific Research (C), JSPS KAKENHI Grant Number JP 19K03542.

\par

\end{document}